\documentclass[a4paper, 11pt, bibliography=totocnumbered]{scrartcl}
\usepackage[ngerman, english]{babel}      
\usepackage[utf8]{inputenc}      
\usepackage{tikz}
\usepackage{svg}
\usepackage{amsthm}
\usepackage{mathtools}
\usepackage{sectsty} 
\usepackage{fancyhdr}
\usepackage{floatpag}
\usepackage{microtype}           
\usepackage[T1]{fontenc}         
\usepackage{lmodern}             
\usepackage{textcomp}            
\usepackage{verbatim}            
\usepackage{color}               
\usepackage{enumitem}            
\usepackage{graphicx}            
\usepackage[hyphens]{url}        
\usepackage{multirow}            
\usepackage{rotating}            
\usepackage{nicefrac}            
\usepackage[hidelinks]{hyperref} 
\usepackage{todonotes}           

\usepackage{wrapfig}             
\usepackage{ragged2e}            
\usepackage{float}               

\usepackage{latexsym,exscale,stmaryrd,amssymb,amsmath}
\usepackage{color}
\usepackage{capt-of}
\usepackage{verbatim}
\usepackage{blindtext}
\usepackage{comment}
\usepackage{amsthm}
\usepackage[utf8]{inputenc}
\usepackage[english]{babel}
\usepackage{ifthen}
\usepackage{listings}
\newtheorem{theorem}{Theorem}[section]

\newtheorem{lemma}[theorem]{Lemma}

\theoremstyle{definition}

\newtheorem{algorithm}[theorem]{Algorithm}

\newtheorem{remark}[theorem]{Remark}

\usepackage[comma,numbers,sort&compress]{natbib}
\usepackage{geometry}

\usepackage{xpatch}
\usepackage{abstract}
\usepackage[section]{placeins}

\newcommand\mycom[2]{\genfrac{}{}{0pt}{}{#1}{#2}}
\newcommand{\argmax}{\mathop{\mathrm{argmax}}}
\newcommand{\argmin}{\mathop{\mathrm{argmin}}}
\chapterfont{\rmfamily\LARGE} 
\newcommand{\s}{\sigma} 

\begin{document}
\setcounter{page}{1}
\vskip 2pc\goodbreak
\bigskip
\bigskip
\centerline{\Large{Deterministic Sparse Sublinear FFT with Improved Numerical Stability}}            

\bigskip
\bigskip
\centerline {\large Gerlind Plonka\footnotemark \,    and  Therese von Wulffen\footnotemark[\value{footnote}]}
\footnotetext{University of G\"ottingen, Institute for Numerical and Applied Mathematics, Lotzestra{\ss}e 16-18, 37083 G\"ottingen, Germany.\\ Email: 
{plonka@math.uni-goettingen.de},
{therese.vonwulffen@stud.uni-goettingen.de}}
\bigskip
\centerline{\large\today}
\bigskip
\bigskip

\begin{abstract}
In this paper we extend the deterministic sublinear FFT algorithm in \cite{PWC18} for fast reconstruction of $M$-sparse vectors ${\mathbf x}$ of length $N= 2^J$, where we assume that all components of the discrete Fourier transform  $\hat{\mathbf x}= {\mathbf F}_{N} {\mathbf x}$ are available.
The sparsity of ${\mathbf x}$ needs not to be known a priori, but is determined by the algorithm. If the sparsity $M$ is larger than $2^{J/2}$, then the algorithm turns into a usual FFT algorithm with runtime ${\mathcal O}(N \log N)$. For  $M^{2} < N$, the runtime of the algorithm is ${\mathcal O}(M^2 \, \log N)$.
The proposed modifications of the approach in \cite{PWC18} lead to a significant improvement of the condition numbers of the Vandermonde matrices which are employed in the iterative reconstruction. 
Our numerical  experiments show that our modification has a huge impact on the stability of the algorithm.
While the algorithm in \cite{PWC18} starts to be unreliable for $M>20$ because of numerical instabilities, the modified algorithm is still numerically stable for $M=200$.
\smallskip


\end{abstract}

\section*{Declarations}

\noindent
\textbf{Funding:} The authors gratefully acknowledge  the support by the German Research Foundation in the framework of the RTG 2088.
\smallskip

\noindent
\textbf{Conflicts of interest/Competing interests:} Not applicable
\smallskip

\noindent
\textbf{Availability of data and material:} Not applicable
\smallskip

\noindent
\textbf{Code availability:} A Python implementation of the new algorithm is available under the link ``software'' on our homepage
\url{http://na.math.uni-goettingen.de}.

\section{Introduction}

Sparse FFT methods can be used in many different applications, where it is a priori known that the resulting signal in time/space or frequency domain is sparse. 
Such algorithms have earned a considerable interest within the last years. 

Many deterministic sparse FFT algorithms are based on combinatorial approaches or  phase shift, see e.g.\ \cite{Ak14,Bi17,Iw10,Iw13,SI13,CLW16}. 
These approaches usually need access to arbitrary values of a given function $f(x) =\sum_{j=1}^M a_j \, {\mathrm e}^{2 \pi {\mathrm i} w_j x}$ assuming that the unknown frequencies $w_j$ are in $[-N/2, N/2) \cap {\mathbb Z}$. The sparse FFT techniques in \cite{HKP13,PTV16} are based on Prony's method.

By contrast, the deterministic algorithms proposed in \cite{PW16,PW17,PWC18,MZI19}, or in  \cite{PPS18}, Section 5.4,  consider the fully discrete problem, where for a given vector ${\mathbf x} \in {\mathbb C}^N$, we want to efficiently compute its discrete Fourier transform ${\hat{\mathbf x}}$ under the assumption that $\hat {\mathbf x}$ is $M$-sparse or has a short support of length $M$. Recently, these techniques have also been transferred to derive sparse fast algorithms for the discrete cosine transform, \cite{BP19,BP19-1}.
\smallskip

\textbf{Problem statement.} 
Let ${\mathbf x} = (x_j)_{j=0}^{N-1} \in {\mathbb C}^N$ with $N= 2^J$ for some $J>1$. Further, let  ${\mathbf F}_N :=(\omega_N^{jk})_{j,k=0}^{N-1} \in {\mathbb C}^{N \times N}$ with $\omega_N := {\mathrm e}^{-2 \pi {\mathrm i}/N}$ denote the Fourier matrix of order $N$,  and  ${\mathbf F}_N^{-1} = \frac{1}{N} 
\overline{{\mathbf F}}_{N}$. We consider the following two  scenarios, which can essentially be treated with the same algorithm.\\[1ex]
(a) Assume that $\hat{\mathbf x} :={\mathbf F}_N \, {\mathbf x} = (\hat{x}_k)_{k=0}^{N-1}$ is given. How  do we, in a sublinear way, determine ${\mathbf x}$ from $\hat{\mathbf x}$,  if it can be assumed that  ${\mathbf x}$ is $M$-sparse with $M^2 < N$?\\[1ex]
(b) Assume that ${\mathbf x} \in {\mathbb C}^N$ is given.
How  do we, in a sublinear way, determine $\hat{\mathbf x} = {\mathbf F}_N {\mathbf x}$ from ${\mathbf x}$, if it can be assumed that  $\hat{\mathbf x}$ is $M$-sparse with $M^2 < N$?\\[1ex]

In both scenarios, $M$  needs not to be known beforehand. However, if $M$ is known, then this knowledge can be used to simplify the algorithm. Throughout the paper, we say that a vector ${\mathbf x}$ is $M$-sparse, if only $M$ components have an amplitude that exceeds  a predetermined small threshold $\epsilon>0$. 

This paper is organized as follows.
In Section 2, we summarize the basic multi-scale idea of the algorithm used in \cite{PWC18} for the scenario (a). Section 3 is devoted to the extension of the method in \cite{PWC18}. First, we present the general pseudocode  of the sparse FFT algorithm.  The numerical stability of this algorithm mainly depends on the condition number of special Vandermonde matrices, which are used at each iteration step for solving a linear system  with at most $M$ unknowns.
In Section \ref{sec:3.1} we give an estimate of the condition number  of the occurring Vandermonde matrices, which are partial matrices of the Fourier matrix. This estimate is used in the sequel to determine the two  free parameters determining the Vandermonde matrix. One parameter stretches the given nodes generating the Vandermonde matrix, and the second parameter determines the number of its rows.
  In Section 4  we briefly show, how the derived algorithm can be simply adapted to solve the sparse FFT problem (b).
Finally, in Section \ref{results} we   present the large impact of the new approach that allows rectangular Vandermonde matrices.
A Python implementation of the new algorithm is available under the link ``software'' on our homepage
\url{http://na.math.uni-goettingen.de}.

\section{Multi-scale Sparse Sublinear FFT Algorithm from \cite{PWC18}}
\label{sec2}

We consider the problem stated in (a) to derive an iterative
stable procedure to reconstruct $\mathbf{x}$ from adaptively chosen Fourier entries of $\hat{\mathbf  x}$. To state the multi-scale algorithm from \cite{PWC18}, we need to define
the periodized vectors
\begin{equation}
\label{def_period}
\mathbf{x}^{(j)} = (x_{k}^{(j)})^{{2}^j-1}_{k=0} :=\Big(\sum_{{l}=0}^{2^{{J}-{j}}-1}x_{{k}+2^{{j}}{l}}\Big)_{k=0}^{2^{j}-1}\in \mathbb{C}^{2^{{j}}},  \qquad  {j}=0,\ldots,{J}.
\end{equation}
In particular, $\mathbf{x}^{(J)}=\mathbf{x}$ and $\mathbf{x}^{(0)}=\sum\limits_{k=0}^{{N}-1}x_{{k}}$   is the sum of all components $\mathbf{x}$. 
Observe that, if the vector $\hat{\mathbf x} = (\hat{x}_k)_{k=0}^{N-1}$ is known, then also the Fourier transformed vectors $\hat{\mathbf x}^{(j)}$ are immediately known, and we have
$$ \hat{\mathbf x}^{(j)} = {\mathbf F}_{2^j} {\mathbf x}^{(j)}  = (\hat{x}_{2^{J-j}k})_{k=0}^{2^j-1} $$
(see Lemma 2.1 in \cite{PW16}).
Throughout the paper, we assume that no cancellation
appears in the periodic vectors, i.e., for each significant component $|{x}_k| > \epsilon$ of $\mathbf{x}$,  $k \in \{0, \ldots, N-1\}$, we have
\begin{equation}
\label{no_cancell}
 |x_{k'}^{(j)}| > \epsilon  \text{ for all } {j}=0,\ldots,{J}-1 , \qquad k' = k \, \mathrm{mod} \, 2^{j}
\end{equation}
for a fixed shrinkage constant $\epsilon>0$. 
 Condition (\ref{no_cancell}) is for example satisfied if all components of ${\mathbf x}$ lie in one quadrant of the complex plane, e.g. $\text{Re}\,  x_{j} \ge 0$ and $\text{Im} \, x_{j} \ge 0$ for $j=1, \ldots , N-1$. 

\noindent\textbf{Idea of the algorithm.} The multi-scale algorithm in \cite{PWC18} iteratively computes ${\mathbf x}^{(j+1)}$ from ${\mathbf x}^{(j)}$, for $j=j_0, \ldots , J-1$.
If the sparsity $M$ of ${\mathbf x}$ is unknown, then we start with $j_0=0$ and ${\mathbf x}^{(0)} := \hat{x}_0 = \sum_{k=0}^N x_k$.
If $M$ with $M^2 < N$ is known beforehand, then we fix $j_0 = \lfloor \log_2 M \rfloor +1$ and compute 
$${\mathbf x}^{(j_0)}:= {\mathbf F}_{2^{j_0}}^{-1} \hat{\mathbf x}^{(j_0)}
= \frac{1}{2^{j_0}} \overline{\mathbf F}_{2^{j_0}} \, (\hat{x}_{2^{J-j_0}k})_{k=0}^{2^{j_0}-1}$$
using an FFT algorithm with complexity $\mathcal{O}(j_0 \, 2^{j_0})
= \mathcal{O}(M\, \log M)
$.
 At the $j$-th iteration step, we assume that ${\mathbf x}^{(j)} \in {\mathbb C}^{2^{j}}$ with sparsity $M_{j}$ has already been computed.  Then we always have $M_j \le M$. For  $M_j^2 < 2^j$, 
the computation of  ${\mathbf x}^{(j+1)}$ from ${\mathbf x}^{(j)}$ is based on the following theorem (see Theorem 2.2 in  \cite{PWC18}).

\begin{theorem}\label{theo1}
Let $\mathbf{x}^{(j)}$,  $j=0,\ldots ,J-1$, be the vectors defined in $(\ref{def_period})$ satisfying $(\ref{no_cancell})$. Then, for each $j=0,\ldots ,J-1$, we have: if $\mathbf{x}^{(j)} \in \mathbb{C}^{2^j}$ is $M_j$-sparse with support
indices  $0\leq\ n_{1}<n_{2}<\ldots<n_{{M}_j} \leq\ 2^{j}-1$, then the vector  $\mathbf{x}^{(j+1)}$ can be uniquely
recovered from $\mathbf{x}^{(j)}$ and $M_{j}$ components
 $\hat{x}_{k_{1}},\ldots,\hat{x}_{k_{M_j}}$ of $\hat{\mathbf{x}}=\mathbf{F}_{N}\, \mathbf{x}$, where the indices 
 $k_{1},\ldots,k_{M_{j}}$ are taken from the set
$\{2^{J-j-1}(2l+1):l=0,\ldots2^j-1\}$ such that the
matrix
\begin{equation}\label{A}
 \mathbf{A}^{(j)}:=\left(\omega_{{N}}^{k_{p} n_{r}}\right)_{p,r=1}^{M_{j}} \end{equation}
 is invertible.
\end{theorem}
The proof of Theorem \ref{theo1} is constructive.  With the notation  ${\mathbf x}^{(j+1)}= \left( \begin{array}{c} {\mathbf x}_{0}^{(j+1)}\\ {\mathbf x}_{1}^{(j+1)} \end{array} \right)$, i.e., 
$\mathbf{x}_{0}^{(j+1)}:= \Big(x_{\ell}^{(j+1)}\Big)_{\ell=0}^{2^{j}-1} $ and $\mathbf{x}_{1}^{(j+1)}:= \Big(x_{\ell}^{(j+1)}\Big)_{\ell=2^{j}}^{2^{j+1}-1} $, we have from (\ref{def_period})
\begin{equation}
\label{x0+x1}
\mathbf{x}^{(j)}=\mathbf{x}_{0}^{(j+1)}+\mathbf{x}_{1}^{(j+1)}.
\end{equation}
Thus, if $\mathbf{x}^{(j)}$ is known, it suffices to compute ${\mathbf x}_0^{(j+1)}$, while $\mathbf{x}_{1}^{(j+1)}$ then follows from (\ref{x0+x1}). We can now use the factorization of the Fourier matrix ${\mathbf F}_{2^{j+1}}$  (see Equation (5.9) in \cite{PPS18}), and obtain
$$ \left( \!\!\begin{array}{c} (\hat{x}_{2\ell}^{(j+1)})_{\ell=0}^{{2^j-1}}\\
(\hat{x}_{2\ell+1}^{(j+1)})_{\ell=0}^{{2^j-1}} \end{array}\!\!\! \right) = \left(\!\!\!\begin{array}{cc} {\mathbf F}_{2^j} & 
{\mathbf 0}\\ {\mathbf 0} & {\mathbf F}_{2^j} \end{array} \!\!\!\right)\! 
\left( \!\!\!\begin{array}{c} \mathbf{x}_{0}^{(j+1)} + \mathbf{x}_{1}^{(j+1)} \\
{\mathbf W}_{2^j} (\mathbf{x}_{0}^{(j+1)} - \mathbf{x}_{1}^{(j+1)}) \end{array} \!\!\!\right) = \left( \!\!\!\begin{array}{c} {\mathbf F}_{2^j} {\mathbf x}^{(j)} \\
{\mathbf F}_{2^{j}} \, {\mathbf W}_{2^j} (2 {\mathbf x}_0^{(j+1)} - {\mathbf x}^{(j)}) \end{array} \!\!\!\right),
$$
where ${\mathbf W}_{2^j} := \textrm{diag} \, (\omega_{2^{j+1}}^0, \ldots , \omega_{2^{j+1}}^{2^j-1})$,  and ${\mathbf 0}$ denotes the zero matrix of size $2^{j} \times 2^{j}$. Thus, we conclude
\begin{equation}\label{teil}
{\mathbf F}_{2^{j}} \, {\mathbf W}_{2^j} \Big( 2{\mathbf x}_0^{(j+1)} - {\mathbf x}^{(j)} \Big) =  \Big( \hat{x}_{2\ell+1}^{(j+1)}\Big)_{\ell=0}^{{2^j-1}}. 
\end{equation}

 Further, (\ref{x0+x1}) together with (\ref{no_cancell}) implies that ${\mathbf x}_0^{(j+1)}$ can only have significant entries for the same index set as ${\mathbf x}^{(j)}$, and we have to compute only these $M_j$ entries.
Introducing the restricted vectors 
\begin{equation} 
    \tilde{\mathbf{x}}^{(j+1)}_{0}:= \left(x^{(j+1)}_{n_{r}}\right)_{r=1}^{M_{j}}\in \mathbb{C}^{M_{j}}, \ \ \ 
    \tilde{\mathbf{x}}^{(j)}:= \left(x^{(j)}_{n_{r}}\right)_{r=1}^{M_{j}}\in \mathbb{C}^{{M_{j}}}, 
    \notag
\end{equation}
we can also restrict the matrix ${\mathbf F}_{2^{j}} \, {\mathbf W}_{2^j}  \in {\mathbb C}^{2^j \times 2^j}$ in the linear system (\ref{teil}) to its $M_j$ columns with indices $n_r$. Finally, it suffices to restrict the system in (\ref{teil}) to $M_j$ linear independent rows, and ${\mathbf x}_0^{(j+1)} $ can still be uniquely computed.
Therefore a restriction ${\mathbf A}^{(j)} \in {\mathbb C}^{M_j \times M_j}$ of the product ${\mathbf F}_{2^{j}} \, {\mathbf W}_{2^j}$
can be chosen as 
\begin{equation}\label{aa}  \mathbf{A}^{(j)}:= \left(\omega_{{2^j}}^{h_p n_{r}}\right)_{p,r=1}^{M_{j}} \, \textrm{diag} \left( \omega_{2^{j+1}}^{n_1}, \ldots ,\omega_{2^{j+1}}^{n_{M_j}} \right).
\end{equation}
Here, the matrix $\left(\omega_{{2^j}}^{h_p n_{r}}\right)_{p,r=1}^{M_{j}}$ is a restriction of ${\mathbf F}_{2^j}$ to the the rows 
 $0 \le h_1 < h_2 < \ldots < h_{M_j} \le 2^j$ and columns  $n_r$, $r=1, \ldots , M_j$  corresponding to support indices of ${\mathbf x}^{(j)}$. 
 The diagonal matrix  is the restriction of ${\mathbf W}_{2^j}$ to the rows and columns $n_r$. Comparison with (\ref{A}) yields $k_p = 2^{J-j-1}(2h_p+1)$, $p=1, \ldots, M_j$.
In Algorithm 2.3 in \cite{PWC18}, Theorem \ref{theo1} is applied to iteratively compute ${\mathbf x}^{(j+1)}$ from ${\mathbf x}^{(j)}$, if 
solving the restricted linear system 
\begin{equation}\label{eqA} {\mathbf A}^{(j)} \Big( 2\tilde{\mathbf x}_0^{(j+1)} - \tilde{\mathbf x}^{(j)} \Big) =  \Big( \hat{x}_{2h_p+1}^{(j+1)}\Big)_{p=1}^{M_j}
\end{equation}
is cheaper than an FFT algorithm for vectors of length $2^j$.

The further results in  \cite{PWC18}  focus on finding good choices of indices  $(h_p)_{p=1}^{M_j}$ at each iteration step.
Thereby, the paper restricts to matrices $\mathbf{A}^{(j)}$ of the form 
\begin{equation}\label{AV} \mathbf{A}^{(j)}:= \left(\omega_{{2^j}}^{\sigma_j \, p \, n_{r}}\right)_{p=0,r=1}^{M_{j}-1,M_j} \, \textrm{diag} \left( \omega_{2^{j+1}}^{{n_1}}, \ldots \omega_{2^{j+1}}^{n_{M_j}} \right),
\end{equation}
i.e., we choose $h_{p+1} =  \sigma_j p$ for $p=0, \ldots , M_j-1$  and some parameter $\sigma_{j} \in \{1, \ldots , 2^{j}-1\}$. The first matrix in the factorization (\ref{AV}) is a Vandermonde matrix generated by the roots of unity $w_{2^j}^{\sigma_j  n_r}$, $r=1, \ldots , M_j$. The iterative algorithm which is based on Theorem \ref{theo1} will be stable, if the linear system (\ref{eqA}) can be efficiently computed in a stable way at each level $j=j_0, \ldots , J$.
Therefore, \cite{PWC18} tries to find parameters $\sigma_j \in \{ 1, \ldots , 2^j-1\}$ such that 
$$ {\mathbf V}_{M_j}(\sigma_j): = \Big( \omega_{2^j}^{\sigma_j \, p \, n_r} \Big)_{p=0,r=1}^{M_j-1, M_j}$$
is invertible and has a good condition  number.
Observe that ${\mathbf V}_{M_j}(\sigma_j)$ is always invertible if we choose  $\sigma_j=1$.
However, $\sigma_j =1$ can lead to a very bad condition number of ${\mathbf V}_{M_j}(\sigma_j)$ and $\mathbf{A}^{(j)}$, respectively.

\begin{remark}
Using Theorem \ref{theo1}, the reconstruction algorithm is based on the idea to iteratively compute periodizations ${\mathbf x}^{(j)} \in {\mathbb C}^{2^{j}}$ of ${\mathbf x} \in {\mathbb C}^{2^{J}}$ of growing length $2^{j}$. At each iteration step, we rigorously exploit the sparsity of these vectors ${\mathbf x}^{(j)}$  and conclude from the support $\{n_{1}, \ldots , n_{M_{j}} \}$ of ${\mathbf x}^{(j)}$ that the support set of ${\mathbf x}^{(j+1)}$ can only be a subset of $\{n_{1}, \ldots , n_{M_{j}} \} \cup \{n_{1}+2^{j}, \ldots , n_{M_{j}}+2^{j} \} $. 
Therefore, the assumption (\ref{no_cancell}) is crucial, since otherwise, not all support indices may be found.

If the sparsity $M$ of ${\mathbf x}$ is known beforehand, then the iteration would start by computing the periodization ${\mathbf x}^{(j_{0})}$ of length $2^{j_{0}} > M$ directly, and we can compare the sparsity of ${\mathbf x}^{(j_{0})}$ with $M$ to ensure that no cancellation appears.
If the sparsity of ${\mathbf x}^{(j_{0})}$ is smaller than $M$, we could then employ a direct FFT algorithm to find the next periodizations ${\mathbf x}^{(j)}$, $j >j_{0}$, until the sparsity of ${\mathbf x}^{(j)}$ is equal to $M$. The complexity of the algorithm would then increase and depends on the level, where the last cancellation appears. In the worst case, if cancellation appears already in ${\mathbf x}^{J-1}$, we would get the complexity of a usual FFT algorithm.
\end{remark}

\section{Extension of the Sparse FFT Algorithm}

The main contribution of this paper is an extension of the algorithm  proposed in \cite{PWC18}, which tremendously improves the stability  of that algorithm to make it really applicable.

We will stay with the  iterative approach to compute ${\mathbf x}^{(j+1)} \in {\mathbb C}^{2^{j+1}}$ from the $M_{j}$-sparse vector ${\mathbf x}^{(j)} \in {\mathbb C}^{2^{j}}$ via (\ref{eqA}) and (\ref{x0+x1}), where we
consider only matrices ${\mathbf A}^{(j)}$, which are given as a product of a Vandermonde matrix  and a diagonal matrix (with condition number $1$) as in  (\ref{AV}), and we will also try to find a suitable parameter  $\sigma_j \in \{1, \ldots , 2^{j}-1 \}$ to improve the numerical stability of the system. The Vandermonde structure provides the advantage that the system in (\ref{eqA}) can be solved with computational cost of  ${\mathcal O}(M^2)$  (see, e.g., \cite{De89}).

We however do not insist on a square matrix  as in \cite{PWC18}, but allow the Vandermonde matrix factor to be a rectangular matrix with more rows than columns of the form 
\begin{equation}\label{VR}
{\mathbf V}_{M_j',M_j}(\sigma_j): = \Big( \omega_{2^j}^{\sigma_j \, p \, n_r} \Big)_{p=0,r=1}^{M_j'-1, M_j}, \qquad M_j' \ge M_j.
\end{equation}
We will choose the  number of rows of the Vandermonde matrix 
${\mathbf V}_{M_j',M_j}(\sigma_j)$ adaptively at each iteration step  
based on the obtained estimate of the condition number of ${\mathbf V}_{M_j',M_j}(\sigma_j)$, where
\begin{equation}
 \kappa_{2}({\mathbf V}_{M',M}(\sigma)) := \frac{ \max_{{\mathbf u}\in \mathbb{C}^{M}, \|\mathbf{u}\|_{2}=1} \| \mathbf{V}_{M',M}(\sigma) \, \mathbf{u} \|_{2}}
{\min_{{\mathbf u }\in {\mathbb C}^{M}, \|\mathbf{u} \|_{2}=1} \|\mathbf{V}_{M',M}(\sigma)\, {\mathbf u}\|_{2}}. 
\end{equation}

We start with presenting the general pseudo code for the case of unknown sparsity $M$. In the further subsections, we will particularly present, how the   matrix ${\mathbf A}^{(j)}$ needs to be chosen, where we allow now a rectangular matrix.
 In Algorithm \ref{algorithm_Rec},  we use the  set notation $I^{(j)} + 2^{j}:= \{n+2^{j}: \, n \in I^{(j)} \}$. 

\begin{algorithm}{\textbf{Sparse (inverse) FFT for unknown sparsity $M$}}
\label{algorithm_Rec}
\newline
\noindent\textbf{{Input:}}
$N=2^{J}$ (length of the vector $\mathbf{x}$), 
\begin{addmargin}[40pt]{0pt} $\epsilon$ (shrinkage constant),\\
possible access to Fourier values $\hat{x}_{k}$, $k=0,\ldots, N-1$.\end{addmargin}

\noindent\textbf{Initialization:}\\
if $|\hat{x}_{0}|<\epsilon$, Output: $M=0$, $\mathbf{x}=\mathbf{0}$,  $I^{(J)}=\emptyset$.\\
if $|\hat{x}_{0}|\geq\epsilon$, then $M:=1$,  $I^{(0)}:=\{0\}$, and $\tilde{\mathbf{x}}^{(0)}=\hat{x}_{0}$.

\par \noindent\textbf{Loop}\\
for $j=0,\ldots,J-1:$
\begin{addmargin}[25pt]{0pt}
if $M^{2}\geq2^{j}$, then
    \begin{addmargin}[25pt]{0pt}
    \textbf{Determine $ \mathbf{x}^{(j+1)}_0$:}\\
    \noindent Put $\hat{\mathbf{z}}^{(j+1)}:= \left(\hat{x}^{(j+1)}_{2p+1}\right)_{p=0}^{2^{j}-1} = \left(\hat{x}_{2^{J-j-1}(2p+1)}\right)_{p=0}^{2^{j}-1}\in \mathbb{C}^{2^{j}}.$\\
    Compute $\mathbf{x}_{0}^{(j+1)} :=\frac{1}{2}\left(\text{diag}\left((\omega_{2^{j+1}}^{k})_{k=0}^{2^{j}-1}\right)^{*}\left(\mathbf{F}_{2^{j}}\right)^{-1}\hat{\mathbf{z}}^{(j+1)}+\mathbf{x}^{(j)} \right)$ using an FFT algorithm.\\ \\
    \textbf{Determine $ \mathbf{x}^{(j+1)}$ and $I^{(j+1)}$:}\\
    Compute $\mathbf{x}_{1}^{(j+1)} :=\mathbf{x}^{(j)}-\mathbf{x}_{0}^{(j+1)}$.\\
    Put ${\mathbf{x}}^{(j+1)} :=\left(({\mathbf{x}}^{(j+1)}_{0})^{T}, ({\mathbf{x}}^{(j+1)}_{1})^{T}\right)^{T}$.\\
   Determine the index set $I^{(j+1)}$ by deleting
all indices in $\left(I^{(j)} \cup (I^{(j)}+2^{j})\right)$\\
    that correspond to entries in ${\mathbf{x}}^{(j+1)}$ with modulus being smaller than $\epsilon$.\\
    Set $M:=\# I^{(j+1)}$.\\
    \end{addmargin}
else
    \begin{addmargin}[25pt]{0pt}
    Set $\tilde{\mathbf{x}}^{(j)}=(\mathbf{x}^{(j)}_l)_{l\in I^{(j)}}$.\\ 
    \textbf{Determine the Matrix $\mathbf{A}^{(j)} \in {\mathbb C}^{M' \times M}$ and the index set $\{h_{p_1}, \ldots , h_{p_M}\}$: see Sections 3.2 and 3.3.}\\
    \textbf{Determine $ \tilde{\mathbf{x}}^{(j+1)}_0$:}\\
    Choose the Fourier values $\hat{\mathbf{z}}^{(j+1)}:= \Big( \hat{x}_{2h_p+1}^{(j+1)}\Big)_{p=1}^{M'} = \Big( \hat{x}_{2^{J-j-1}(2h_p+1)}\Big)_{p=1}^{M'}$.\\
    Compute $\tilde{\mathbf x}_0^{(j+1)}$ by solving the system
    \begin{equation}\label{a*}
    {\mathbf A}^{(j)} \Big( 2\tilde{\mathbf x}_0^{(j+1)} - \tilde{\mathbf x}^{(j)} \Big) =  \hat{{\mathbf z}}^{(j+1)}.
    \end{equation}
    \textbf{Determine $ \Tilde{\mathbf{x}}^{(j+1)}$ and $I^{(j+1)}$:}\\
    Compute $
    \tilde{\mathbf{x}}^{(j+1)}_{1} :=\Tilde{\mathbf{x}}^{(j)}-\tilde{\mathbf{x}}^{(j+1)}_{0}
    $.\\
    Put $\tilde{\mathbf x}^{(j+1)} :=\left((\Tilde{\mathbf{x}}^{(j+1)}_{0})^{T}, (\Tilde{\mathbf{x}}^{(j+1)}_{1})^{T}\right)^{T}
    $.\\
    Determine the index set $I^{(j+1)}$ by deleting
all indices in $\left(I^{(j)} \cup (I^{(j)}+2^{j})\right)$
    that correspond to entries in $\tilde{\mathbf{x}}^{(j+1)}$ with modulus being smaller than $\epsilon$.\\
   Set $M:=\# I^{(j+1)}$.\\
    \end{addmargin}
\end{addmargin}

\noindent\textbf{Output:}
$I^{(J)}$, the set of active indices in  of  $\mathbf{x}$,
\begin{addmargin}[50pt]{0pt} $\Tilde{\mathbf{x}}=\Tilde{\mathbf{x}}^{(J)}=(x_{l})_{l \in I^{(J)}}$, the vector restricted to nonzero entries.
\end{addmargin}

\end{algorithm}

To determine the suitable matrix 
$${\mathbf A}^{(j)} ={\mathbf V}_{M_j',M_j}(\sigma_j) \, \textrm{diag} \left( \omega_{2^{j+1}}^{{n_1}}, \ldots \omega_{2^{j+1}}^{n_{M_j}} \right), $$
we have to find a well-conditioned Vandermonde matrix ${\mathbf V}_{M_j',M_j}(\sigma_j)$.
Our procedure consists of two steps. \\
1) We compute a suitable parameter $\sigma_j$ with ${\mathcal O}(M^2)$ operations.\\ 
2) We compute the number $M_j'$ of needed rows in the Vandermonde matrix, to achieve a well-conditioned coefficient matrix in the system (\ref{a*}).

As seen already in \cite{PWC18}, we can simplify the procedure of determining ${\mathbf V}_{M_j',M_j}(\sigma_j)$, if the number of significant entries $M_j$ of ${\mathbf x}^{(j)}$ did not change in the previous iteration step, i.e., if $M_{j-1}=M_j$.
In this case, we can just choose $\sigma_{j+1} := 2 \sigma_j$ and stay with the number of columns, i.e., $M_{j}' := M_{j-1}'$  (see also Subsection \ref{sec:mj}).

\subsection{Estimation of the condition number of ${\mathbf V}_{M_j',M_j}(\sigma_j)$}
\label{sec:3.1}

It is crucial for our algorithm to have a good estimate of the condition  number of ${\mathbf V}_{M_j',M_j}(\sigma_j)$.
The condition  number of ${\mathbf V}_{M_j',M_j}(\sigma_j)$ strongly depends on the minimal distance between its generating nodes $\omega_{2^j}^{\sigma_j n_r}$.
 More precisely, we have the following theorem  (see \cite{Mo15,PWC18} or Theorem 10.23 in \cite{PPS18}).
\begin{theorem}
\label{Theorem_Moitra}
Let  $0\leq n_{1}<n_{2}<\ldots<n_{M_j}<2^{j}$  be a given set of indices. For a
given $\sigma_j \in \{1,\ldots,2^{j}-1\}$ we define
\begin{equation}
\label{def_d_sigma}
 d_{j} =   d({\sigma_j}) := \min_{1 \leq k<l \leq M_j}  \left( (\pm \sigma_j \, (n_{l}-n_{k})) \, \mathrm{mod}\, 2^{j}  \right)
\end{equation} 
as the smallest (periodic) distance between two indices $\sigma_j \,{n_l}$ and $\sigma_j \, {n_k}$, and assume that $d_{j}>0$.
 Then the condition number $\kappa_{2}(\mathbf{V}_{M_j',M_j}(\sigma_j))$ of the Vandermonde matrix $\mathbf{V}_{M_j',M_j}(\sigma_j):= \Big( \omega_{2^{j}}^{\sigma_j \, p \, n_r} \Big)_{p=0,r=1}^{M_j'-1, M_j}$ satisfies
\begin{equation}
\label{upperbound_condition}
    \kappa_{2}(\mathbf{V}_{M_j',M_j}(\sigma_j))^2 \leq \frac{M_j'+2^{j}/d_{j}}{M_j'-2^{j}/d_{j}},
\end{equation} 
provided that $M_j'>\frac{2^{j}}{d_{j}}$.
\end{theorem}

However, this estimate cannot be used for square matrices,  i.e., for $M_{j} = M_{j}'$,  and it is not very sharp for large $M_j$.
Indeed, if $d_{j}=2^{j}/M_j$ which means that the values $\sigma_j \, n_k$ are equidistantly distributed  on the periodic interval $[0, 2^{j})$, then the square matrix $M_j^{-1/2} \, \mathbf{V}_{M_j,M_j}(\sigma_j)$ (with $M_j'=M_j$) is orthogonal with condition number $1$ 
 (see \cite{BF07}), while the estimate (\ref{upperbound_condition}) cannot be applied.
On the other hand, if $M_j'=2^{j}$,  then we can simply conclude that
$\mathbf{V}_{2^{j},M_j}(\sigma_j)^* \mathbf{V}_{2^{j},M_j}(\sigma_j) = 2^{j}\, {\mathbf I}_{M_j}$ such that we again achieve condition number $1$, while (\ref{upperbound_condition}) provides $\frac{2^{j} (1+1/d_{j})}{2^{j}(1-1/d_{j})}$, which again fails for the worst case $d_{j}=1$ completely.
Therefore, we apply another estimate, which is a simple consequence of the Theorem of Gershgorin, and can be iteratively computed during the iteration steps. It is based on the following Theorem.

\begin{theorem}\label{cond2}
Let $0\leq n_{1}<n_{2}<\ldots<n_{M_j}<2^{j}$  be a given set of indices,  and assume that $\sigma_{j}(n_{k}-n_{\ell}) \neq 0 \, \mathrm{mod} \, 2^{j}$.
Further, let for all $k=1, \ldots , M_j$,  $M_{j} \le M_{j}' \le 2^{j}$, and
\begin{equation}\label{sk}
S_k(\sigma_j) := \sum_{\mycom{\ell=1}{\ell \neq k}}^{M_j} \left|\frac{\sin \Big( \frac{M_j' \pi }{2^{j}} \, \sigma_j \, (n_k- n_{\ell})\Big)}{\sin \Big(\frac{\pi}{2^{j}} \, \sigma_j \, (n_k- n_{\ell})\Big)} \right|.
\end{equation}
Then the condition number of the Vandermonde matrix $\mathbf{V}_{M_j',M_j}(\sigma_j)$ in $(\ref{VR})$ is bounded by
\begin{equation}
\label{upperbound1}
    \kappa_{2}(\mathbf{V}_{M_j',M_j}(\sigma_j))^2 \leq \frac{M_j'+ \max_{k} S_k(\sigma_j)}{M_j'-\max_{k} S_k(\sigma_j)}.
\end{equation} 
\end{theorem}

\begin{proof}
Considering the matrix product ${\mathbf W}:= \mathbf{V}_{M_j',M_j}(\sigma_j)^* \, \mathbf{V}_{M_j',M_j}(\sigma_j) \in {\mathbb C}^{M_j \times M_j}$, it follows for the components $w_{k,\ell}$ of ${\mathbf W}$ that 
$$ w_{k,k} = \sum_{p=0}^{M_j'-1} \omega_{2^{j}}^{p \, \sigma_j \, (n_k - n_k)} = M_j', \qquad k=0, \ldots , M_{j}-1,$$
and for $k \neq \ell$  and $\sigma_{j}(n_{k}-n_{\ell}) \neq 0 \, \mathrm{mod} \, 2^{j}$,
$$ |w_{k,\ell}| = 
\Big|\sum_{p=0}^{M_j'-1} \omega_{2^{j}}^{p \, \sigma_j \, (n_k - n_{\ell})}\Big| = \Big| \frac{1-\omega_{2^{j}}^{M_j'\sigma_j (n_k - n_{\ell})}}{1- \omega_{2^{j}}^{\sigma_j (n_k - n_{\ell})}} \Big|
= \Big| \frac{\sin(\frac{M_j' \pi}{2^{j}} \, \sigma_j (n_k- n_{\ell}))}{\sin(\frac{\pi}{2^{j}} \, \sigma_j (n_k- n_{\ell}))} \Big|. $$
Thus, $S_k(\sigma_j)$ is the sum of the absolute values of all non-diagonal components in the $k$-th row of ${\mathbf W}$.
The Theorem of Gershgorin implies now that the maximal eigenvalue of ${\mathbf W}$ is bounded from above by $M_j' + \max_{k} S_k(\sigma_j)$, and the smallest  eigenvalue is bounded from below by $M_j' - \max_{k} S_k(\sigma_j)$.
\end{proof}

While the estimate (\ref{upperbound1}) is quite simple to achieve, it is more accurate than (\ref{upperbound_condition}).
In particular, in the two special cases $M_j'=M_j$, $d_{j} = 2^{j}/M_j$ and $M_j'=2^{j}$, $d_{j}=1$, the estimate is sharp, and we obtain the true condition number $1$.

For our computation of $\sigma_{j}$ in Section \ref{sec:sigma}, we will however simplify (\ref{sk}) and  will consider instead an approximation of the upper bound of ${S}_k(\sigma_j)$,
\begin{equation}\label{st}
\tilde{S}_k(\sigma_j) := \sum_{\mycom{\ell=1}{\ell \neq k}}^{M_j} \left|\frac{1}{\sin(\frac{\pi}{2^{j}} \, \sigma_j \, (n_k^{(j)}- n_{\ell}^{(j)}))} \right| \ge S_k(\sigma_j)
\end{equation}
which is not longer dependent on $M_j'$. 
 Note that $\tilde{S}_k(\sigma_j)>2^{j}$ can appear, if $\sigma_{j}$ is not well chosen.

\subsection{Efficient computation of $\sigma_j$}
\label{sec:sigma}
For a given set of indices $0 \le n_1 < n_2 < \ldots   < n_M < 2^j$
we want to find a suitable $\sigma_j \in \{1, \ldots , 2^{j}-1 \}$ such that  an approximation of $\max_k \tilde{S}_k(\sigma_j)$ is minimal. 
More precisely, as shown in Algorithm \ref{algorithm_prime}, 
we compare different possible parameters $\sigma$ by comparing the sums of four terms in the sum (\ref{st}), where the largest term is always included.

We surely could just consider all possible sets $\{ \sigma n_1, \ldots, \sigma n_M\}$ for $\sigma \in \{1, \ldots , 2^{j}-1\}$, compute the maximal sum $\tilde{S}_k^{(j)}(\sigma)$ and compare the results to find the optimal parameter $\tilde\sigma_j$. However, this procedure is too expensive. To achieve a sparse FFT algorithm with the desired overall complexity of ${\mathcal  O}(M^2 \log N)$, we can spend at most ${\mathcal O}(M^2)$ operations to find a suitable parameter $\sigma_j$.

To avoid vanishing  distances $\pm\sigma_j(n_k - n_\ell) \, \mathrm{mod} \,  2^{j}=0$
for all $n_k \neq n_\ell$, we will only consider odd integers $\sigma_j \ge 1$. We then have 
 that $2^{j}$ and $\sigma_j$ are co-prime such that for each odd $\sigma_j$ we at least achieve that $\max_k \tilde{S}_k^{(j)}(\sigma_j)$ is bounded.  As our numerical tests show that prime numbers are good candidates for $\sigma_j$, we propose the following algorithm to determine $\sigma_j$.

 \begin{algorithm}{\textbf{(Computation of $\sigma_j$ if $M_j > M_{j-1}$)}}
 \label{algorithm_prime}
{}\\
\textbf{Input:}\\ $N :=2^j$. \\Index set $I^{(j)}= \{n_1, \ldots , n_{M_j}\}$.\\ 
\textbf{Initialization:}\\
Set $M_{j} :=\#I^{(j)}$ and choose $K$ with $K \le M_{j}/ \log_{2} M_{j}$. \\
Let $\Sigma :=$ be set of $K$ largest prime numbers smaller than $N/2$. \\ 
\textbf{Loop:}\\
For all $\sigma \in \Sigma$:
\begin{addmargin} [25pt]{0pt}
Compute the set $\sigma I^{(j)} := \{\sigma l \mod N :l \in I^{(j)} \}$.\\
Order the elements of $\sigma I^{(j)}$ by size to get $\tilde{n}_1 < \ldots  < \tilde{n}_{M_{j}}$.\\
Compute the sequence of distances $\delta_k:=\tilde{n}_k- \tilde{n}_{k-1}$, $k=1, \ldots , M_{j}$,  where $\tilde{n}_0:=\tilde{n}_{M_{j}}-N$.\\
Find the index of the smallest distance $\tilde{k} := \argmin_{k=1, \ldots , M_{j}} \delta_k$.\\
Compute 
$$ D_{\sigma} := \max \left\{\left|\frac{1}{\sin(\frac{\delta_{\tilde{k}} \pi}{N}) }\right| + \left|\frac{1}{\sin(\frac{\delta_{\tilde{k}-1} \pi}{N}) }\right| , \left|\frac{1}{\sin(\frac{\delta_{\tilde{k}} \pi}{N})}\right|+ \left|\frac{1}{\sin(\frac{\delta_{\tilde{k}+1} \pi}{N}) }\right|    \right\}$$
with the convention that $\delta_0:=\delta_{M_{j}}$ and $\delta_{M_{j}+1}:=\delta_1$.
\end{addmargin}
\textbf{Completion:}\\
Choose $\sigma \in \Sigma$ with minimal  $D_\sigma$.\\
If there are several parameters $\s$ achieving the same value $D_{\sigma}$,
\begin{addmargin} [25pt]{0pt}
choose the $\sigma$ which minimizes the sum
 $\left|\sum_{k=1}^{M_{j}} \omega_N^{\sigma n_k} \right|$.
\end{addmargin}
\textbf{Output:} $\sigma_j :=\sigma$
\end{algorithm}

The most expensive step in Algorithm \ref{algorithm_prime} is the sorting of $M_{j}$ elements in $\sigma I^{(j)}$,  which can be done with $M_{j} \log M_{j} \le M \log M$ operations.  Since $\Sigma$ contains $K < M_{j}/\log_{2} M_{j}$ elements, the algorithm has a computational cost of ${\mathcal O}(M^2)$. 
Note, that we did not compute the complete sum $\tilde{S}_k(\sigma)$ for all choices of $\sigma$ in Algorithm \ref{algorithm_prime}. 
Instead, for fixed $\sigma$, we search for an index $\tilde{k}$ that provides the smallest (periodic) distance $|\sigma (n_{\tilde{k}} - n_{\tilde{k}-1})| = \min_{k \neq \ell}|\sigma (n_{k} - n_{\ell})|$. This index $\tilde{k}$ is a good candidate for $\argmax_k \tilde{S}_k(\sigma)$. We then only 
compute the sum of the largest component and the neighboring component of $\tilde{S}_{\tilde{k}}(\sigma)$ instead of the full sum, since  $\tilde{S}_{\tilde{k}}(\sigma)$ is mainly governed by these components.

\begin{remark}
Using Theorem \ref{Theorem_Moitra} it is of course also possible to determine $\sigma_j$ by comparing only the minimal distance $d({\sigma})$ in (\ref{def_d_sigma}) for all $\sigma \in \Sigma$, and to choose $\sigma \in \Sigma$ that maximizes this distance.\\
There are always enough odd prime numbers available in $[1, \, \frac{2^{j}}{2}]$, since $M_j^2 < 2^{j}$  (see, e.g., \cite{RS62}).
\end{remark}

\subsection{Determination of $M_j'$}

Further, we need to fix the number of needed rows $M_j' \ge M_j$ to ensure that the Vandermonde matrix $\mathbf{V}_{M_j',M_j}(\sigma_j)$ is well conditioned.
Employing Theorem \ref{cond2},  we consider $M_{j}' = c \, M_{j}$ for a small set of integers $c$, e.g. $c \in \{1, 2,  5\}$. Starting with $c=1$, we compute $\max_{k} S_{k}(\sigma_{j})$ in (\ref{sk}) with ${\mathcal O}(M_{j}^{2})$ operations, and check via (\ref{upperbound1}) whether  the condition number of ${\mathbf V}_{M_{j}',M_{j}}(\sigma_{j})$  is acceptable. If it is too large, we enlarge $c$.

\begin{remark}
We  can also use the estimates in Theorem \ref{Theorem_Moitra} for determining $M_j'$. In this case, we simply fix $M_j'$ such that 
$$ \left( \frac{M_j' + 2^{j}/d_{j}}{M_j'- 2^{j}/d_{j}} \right)^{1/2} < C $$
where $C$ is a pre-determined bound for the condition number of ${\mathbf V}_{M_{j}',M_{j}}(\sigma_{j})$.
However, this estimate  usually leads to a strong overestimation of $M_j'$.
\end{remark}

In our numerical experiments  we achieved good results with the simple bound 
\begin{equation}\label{res} M_j' = c\,  M_j \qquad \textrm{with} \qquad c :=\min\left\{\left\lfloor \frac{2^j/M_j}{d_{j}}\right\rfloor, c_{\max}\right\},
\end{equation}
where $c_{\max}$ is usually an integer with $c_{\max} \le 5$  (see Section \ref{results}).
This setting  can also be understood as a compromise for having a good condition  number of the  
matrix ${\mathbf A}^{(j)}$ in the system (\ref{a*}) on the one hand and the computational cost 
to solve the linear system on the other hand. Using for example the QR decomposition algorithm in \cite{De89} 
for rectangular Vandermonde matrices of size $c M_j \times M_j$, we obtain a complexity of $(5c+ \frac{7}{2})M_j^2 + {\mathcal O}(c M_j)$.

\subsection{Choice of ${\mathbf A}^{(j)}$ if $M_{j-1}=M_j$}
\label{sec:mj}

If $M_j=M_{j-1}$, we apply the following Lemma which is an extension of Theorem 4.2 in \cite{PWC18}.

\begin{lemma}
Let $\sigma_{j-1}$ and $M_{j-1}'$ be the parameters used in the Algorithm $\ref{algorithm_Rec}$ to determine ${\mathbf V}_{M_{j-1}', M_{j-1}}(\sigma_{j-1})$ in the iteration step $j-1$, where $0<n_1^{(j-1)}<\ldots<n_{M_{j-1}}^{(j-1)}<2^{j-1}$ are the support indices of $\textbf{x}^{(j-1)}$.
Further, assume that we have found ${\mathbf x}^{(j)}$ with $M_j = M_{j-1}$, and  support indices $0<n_1^{(j)}<\ldots<n_{M_{j}}^{(j)}<2^{j}$.
 Then we can simply choose $\sigma_j:=2 \sigma_{j-1}$ and $M_j' := M_{j-1}'$ to achieve a Vandermonde matrix 
${\mathbf V}_{M_{j}', M_{j}}(\sigma_{j})$ for iteration step $j$ of Algorithm $\ref{algorithm_Rec}$. With this choice, ${\mathbf V}_{M_{j}', M_{j}}(\sigma_{j})$ coincides with ${\mathbf V}_{M_{j-1}', M_{j-1}}(\sigma_{j-1})$ up to possible permutation of columns. In particular, we have
$$ \kappa_2({\mathbf V}_{M_{j}', M_{j}}(\sigma_j)) = \kappa_2({\mathbf V}_{M_{j-1}', M_{j-1}}(\sigma_{j-1})).$$
\end{lemma}

\begin{proof} If 
$M_j=M_{j-1}$, then it follows that $n_r^{(j)} \in \{n_r^{(j-1)}, n_r^{(j-1)}+2^{j-1}\}$  for all $r=1, \ldots , M_{j-1}$. With $\sigma_j=2\sigma_{j-1}$ we obtain
$$\sigma_j \, n_r^{(j)} \, \mathrm{mod} \, 2^j = 2\sigma_{j-1}n_r^{(j)} \, \mathrm{mod} \, 2^j=2\sigma_{j-1}n_r^{(j-1)} \, \mathrm{mod} \, 2^j. $$ 
Thus, for $p=1, \ldots, M_j'$ (with $M_j'=M_{j-1}'$),
$$ \omega_{2^j}^{\sigma_j (p-1) n_r^{(j)}} = 
\omega_{2^j}^{2 \sigma_{j-1} (p-1) n_r^{(j)}} =
\omega_{2^j}^{2 \sigma_{j-1} (p-1) n_r^{(j-1)}} =
\omega_{2^{j-1}}^{\sigma_{j-1} (p-1) n_r^{(j-1)}}.
$$
Hence, ${\mathbf V}_{M_{j-1}', M_{j-1}}(\sigma_{j-1})$ and ${\mathbf V}_{M_{j}', M_{j}}(\sigma_{j})$ have the same columns, and may differ only due to a different ordering of columns. In other words, there is an $M_j \times M_j$ permutation matrix ${\mathbf P}_{M_j}$, such that ${\mathbf V}_{M_{j}', M_{j}}(\sigma_{j}) = {\mathbf V}_{M_{j-1}', M_{j-1}}(\sigma_{j-1}) \, {\mathbf P}_{M_j}$.
In particular, the two matrices have the same condition number.
\end{proof}

This observation implies that there will be no extra effort to compute the  matrix ${\mathbf A}^{(j)}$ at all iteration steps $j$, where the sparsity $M_j$ has not changed compared to $M_{j-1}$.

\section{The direct sparse FFT algorithm}

We consider now the direct sparse FFT problem stated in (b) in  Section 1. For given ${\mathbf x} \in {\mathbb C}^N$, we want to determine ${\mathbf y} := \hat {\mathbf x} = {\mathbf F}_N \, {\mathbf x}$, assuming that $ {\mathbf y}$ possesses unknown sparsity $M$.
We will show that our Algorithm \ref{algorithm_Rec} can be transferred to this problem.

First, we observe that the Fourier matrix satisfies the property
$$ {\mathbf F}_N^{-1} = \frac{1}{N} \overline{\mathbf F}_N = \frac{1}{N} {\mathbf J}_N' \, {\mathbf F}_N $$
 (see Equation (3.34) in \cite{PPS18}), where ${\mathbf J}_N' := (\delta_{(j+k) \, \mathrm{mod}\,  N})_{j,k=0}^{N-1}$ is the so-called flip matrix with $({\mathbf J}_N')^{-1} = {\mathbf J}_N'$. Here, $\delta_j$ denotes the Kronecker symbol, i.e.,  $\delta_j =0$ for $j\neq 0$ and $\delta_j=1$ for $j=0$.
Thus, the relation  ${\mathbf x} = {\mathbf F}_N^{-1} \, {\mathbf y}$ is equivalent  to
$$ {\mathbf w} := N \, {\mathbf J}_N' {\mathbf x} = {\mathbf F}_N {\mathbf y}. $$
In other words, if we replace the given vector ${\mathbf x}$ by 
${\mathbf w}$ in Algorithm \ref{algorithm_Rec}, then ${\mathbf w}$ is the given Fourier transform of the desired vector ${\mathbf y}$, and we can apply Algorithm \ref{algorithm_Rec} directly to compute ${\mathbf y}$.

\section{ Numerical experiments}
\label{results}

First, we present some numerical experiments showing that the algorithm in \cite{PWC18} for sparsity $M>20$ is no longer reliable.
We generate randomly chosen sets of support indices $I_M \subset \{0, \ldots , 2^{15}- 1\}$ with different cardinalities $M=20,30,\ldots,100$, and randomly choose values $x_k$ for $k \in I_M$ in double precision arithmetics. 
Then we apply our Algorithm \ref{algorithm_Rec}, where access to the Fourier transform of $\textbf{x}\in\mathbb{C}^{2^J}$ is provided.
While $\sigma_j$ is optimally chosen as a prime number according to 
Algorithm 4.5 in \cite{PWC18}, we only consider square Vandermonde matrices (as in \cite{PWC18}), i.e., we set $c_{\max}=1$. We compare the output index set $I_{out}$ with the generated set $I_M$ of indices and count the failures of 100 tests for each $M$. The results are presented in Figure \ref{error_rate_J15}. The test shows that the algorithm starts to be unreliable for sparsity $M >20$.

\begin{figure}
\begin{center}
    \includegraphics[scale=0.6]{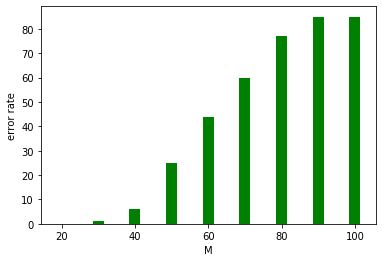}
    \captionof{figure}{Error rate in  percentage for the computed set of indices for $c_{max}=1$ and $J=15$.}
    \label{error_rate_J15}
\end{center}
\end{figure}

We now run the test with the same input data as above, but used the criteria in (\ref{res}) with  $c_{\max}=2$. For any $M=20,30,\ldots,100$, no failures occur  for the computed set of indices $I_{out}$, i.e., we always find $I_M= I_{out}$. Even if we run the tests for $M=200$, the error rate is still zero. 

To understand this strong effect  when the number of rows of the Vandermonde matrix is enlarged, we analyze the condition numbers of the Vandermonde matrices occurring in the computations for different 
values $c_{\max}$.
We generate sets  $I_M$ of indices and randomly choose the amplitudes of components of ${\mathbf x}$ with support  $I_M$. For Algorithm \ref{algorithm_Rec}, we provide access to the  Fourier transformed vector  $\hat{\mathbf x}$ as an input as before for the tuples $(J,M)$ with $J=15,16,\ldots,22$, and $M=20,30,40,50$.  In this experiment, we vary  $c_ {\max} \in \{1,2,5\}$. In each test we compute the average over all condition numbers of the used Vandermonde matrices and repeat this 20 times for each tuple $(J,M)$. Finally, we take the mean of all the 20 averages, and obtain the results given in the Tables \ref{tab:condition_taumax1} and \ref{tab:condition_taumax2und5_bigN}. The results in Table \ref{tab:condition_taumax1} show that a suitable choice of the parameter $\sigma_j$, as applied in \cite{PWC18}, is not sufficient to ensure moderate condition numbers of the Vandermonde matrices involved in the sparse FFT algorithm for $M \ge 20$.

\begin{table}[th]
    \centering
    \begin{tabular}{c}
    $c_{max}=1$\\
    \centering
    \begin{tabular}{|c|r r r r|}
    \hline
    $J$ & $M=20$  & $M=30$ & $M=40$ & $M=50$\\
    \hline
    15 & {45587} & {8959761} & {826581656} &{813444189055}  \\
    16 & {150932} & {3541859} & {41764903} &{535590260990711}   \\
    17 &{502398} &  {1044096} & {2914884097} &{719367030204.95}   \\
    18 & {103809} &  {674572} & {1080286999258065} &{1016723525704275}\\
    19 & { 10491} &  {4832052} & {111942753} &{12377927191183} \\
    20 & {41983} & {711412} & {918528399} &{93462229700} \\
    21 & {61938} & {3502253} & {567002193} &{143672696329261} \\
    22 & {388062} &  {37168024} & {259341688} &{28197228} \\
    \hline
    \end{tabular}
    \end{tabular}
\caption{Average condition  number for $c_{max}=1$ after 20 tests.}
\label{tab:condition_taumax1}
\end{table}

In Table \ref{tab:condition_taumax2und5_bigN}, we provide some further condition numbers for larger numbers $M$ of significant  vector entries up to $M=200$ and $N=2^{15},\dots,2^{22} $. The experiments show that $c_{\max} =2$, i.e., doubling the number of rows in the  matrix ${\mathbf A}^{(j)}$, is usually sufficient for $M\le 100$. For $M>100$, we need to take a larger $c_{\max}$.

\begin{table}[th]
    \centering
    \begin{tabular}{c c }
    $c_{max}=2$ & $c_{max}=5$\\
    \begin{tabular}{|c|r r r|}
    \hline
    $J$ & $M=20$  & $M=100$ & $M=200$\\
    \hline
    15 & {4.31} &{128.83} &{12623.74} \\
    16 & {5.57}  &{415.11} &{167096.38}  \\
    17 & {7.94}  &{74.23} &{32290.12} \\
    18 & {40.52}  &{591.17} &{5901.65}  \\
    19 & {14.76}  &{732.74} &{154631.91}  \\
    20 & {14.46}  &{231.35} &{27979.52}\\
    21 & {17.51} &{259.04} &{14604.35}\\
    22 & {12.86}  &{360.91} &{17897.02} \\
    \hline
    \end{tabular}
    &
    \begin{tabular}{|c|r r r|}
    \hline
    $J$ & $M=20$  & $M=100$ & $M=200$\\
    \hline
    15 & {1.33}  &{4.52} &{16.42} \\
    16 & {1.36} &{8.01} &{38.64}  \\
    17 & {1.43} &{4.97} &{37.78}  \\
    18 & {1.79} &{8.59} &{19.76} \\
    19 & {1.44} &{10.13} &{38.64} \\
    20 & {1.39}  &{9.56} &{28.29} \\
    21 & {1.75}  &{7.25} &{22.41} \\
    22 & {1.63}  &{6.04} &{23.12} \\
    \hline
    \end{tabular}
\end{tabular}
\caption{Average condition  number for $c_{max}=2$ (left) and $c_{max}=5$ (right) after 20 tests.}
\label{tab:condition_taumax2und5_bigN}
\end{table}

\begin{figure}[t]
\begin{center}
    \includegraphics[scale = 0.60]{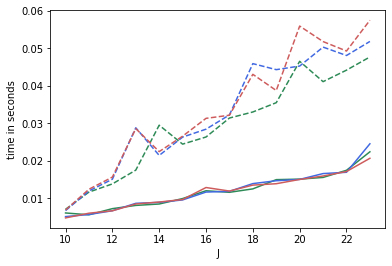} 
    \captionof{figure}{Runtime comparison of the Algorithmus \ref{algorithm_Rec} for $c_{max}=1$ (green), $c_{max}=5$ (blue), $c_{max}=20$ (red) for $M=10$ (solid line) and $M=30$ (dashed line) for length $N=2^J$ with $J=10,\ldots,24$.}
\label{runtime_tau}
\end{center}
\end{figure}

\begin{figure}[t]
\begin{center}
    \includegraphics[scale = 0.60]{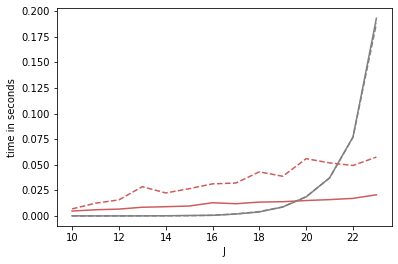}
    \captionof{figure}{Runtime comparison of Algorithmus \ref{algorithm_Rec} and  $c_{\max}=20$ (red) and the FFT (gray) for $M=10$ (solid line) and $M=30$ (dashed line) for length $N=2^J$ with $J=10,\ldots,24$.}
\label{runtime_FFT}
\end{center}
\end{figure}

Now, we investigate how the runtime of the Algorithm depends on   $c_{\max}$.  In Figure \ref{runtime_tau} we present  the average runtime for 20 tests with randomly chosen sparse vectors
with sparsities $M=10,30$ and for $c_{\max}=1, c_{\max}=5,c_{\max}=20$.
As we see in Figure \ref{runtime_tau}, our modifications have only a very small effect
on the runtime. Finally, in Figure \ref{runtime_FFT} we compare the runtime of the  Python implemented FFT \texttt{numpy.fft.fft} of length $2^J$ with our algorithm for $c_{\max}=20$. 
We can see, that our current Python implementation starts to be faster than the FFT for $ M \le 30$ and $N \ge 2^{20}$. It is available  under the link ``software'' on our homepage
\url{http://na.math.uni-goettingen.de}.

\section{Conclusions}
In this paper, we have presented a modification of the sparse FFT algorithm in \cite{PWC18}, which is based on the assumption that the wanted vector ${\mathbf x} \in {\mathbb C}^{N}$ with $N= 2^{J}$ is $M$-sparse, and that the components of the discrete Fourier transform $\hat{\mathbf x}= {\mathbf F}_{N}\, {\mathbf x}$ are available.
Our proposed algorithm has the complexity ${\mathcal O}(M^{2} \log N)$ and is sublinear in $N$ for small $M$.
As in \cite{PWC18}, the reconstruction of ${\mathbf x}$ is based on an iterative reconstruction of $2^{j}$-periodizations of ${\mathbf x}$ for $j=0, \ldots , J$. At each iteration step, one needs to solve an equation system of size ${\mathcal O}(M)$, where the coefficient matrices are governed by Vandermonde matrices which are submatrices of the Fourier matrix ${\mathbf F}_{2^{j}}$. 
Differently from \cite{PWC18}, we have considered rectangular Vandermonde matrices, and we  have presented efficient methods to determine these matrices in dependence of two parameters, which both have a huge impact on the condition number. 
The first parameter $\sigma_{j}$ changes the nodes $\omega_{2^{j}}^{n_{\ell}}$, $\ell=1, \ldots , M_{j}$ determining the Vandermonde matrix  to  $\omega_{2^{j}}^{\sigma_{j} n_{\ell}}$. Here $M_{j}\le M$ denotes the found sparsity of ${\mathbf x}^{(j)}$.  The second parameter $M_{j}' \ge M_{j}$ denotes the number of rows in the Vandermonde matrix.
One ingredient to determine suitable parameters $\sigma_{j}$ and $M_{j}'$ is the new estimate for the condition number of the occurring  Vandermonde matrices in Theorem \ref{cond2}.
As shown in the numerical experiments, the  presented modification of the sparse FFT algorithm makes it applicable also for larger sparsity values $M$ while the original algorithm in \cite{PWC18} started to be unreliable already for $M>20$. 

\section*{Compliance with Ethical Standards}

This article does not contain any studies with human participants or animals performed by any of the authors.\\
Informed Consent: Does not apply

\section*{Acknowledgement}
The authors like to thank the reviewers for very exact reading of the manuscript and many constructive remarks for its improvement.
The authors gratefully acknowledge  the support by the German Research Foundation in the framework of the RTG 2088.

\renewcommand\bibname{ References}
\begingroup
\small
\bibliography{bibliography}
\endgroup
\end{document}